\documentclass[review,12pt,a4paper,3p]{elsarticle}

\usepackage{amsmath}
\usepackage{lipsum,bm}
\usepackage{amsfonts,wasysym}
\usepackage{graphicx}
\usepackage{mathrsfs}
\usepackage{epstopdf}
\usepackage{algorithmic}
\usepackage{amsopn}
\usepackage{amssymb}
\usepackage{enumitem}
\usepackage{url}
\usepackage{hyperref}
\usepackage{xcolor}

\ifpdf
\DeclareGraphicsExtensions{.eps,.pdf,.png,.jpg}
\else
\DeclareGraphicsExtensions{.eps}
\fi

\newdefinition{example}{Example}[section]
\newdefinition{setting}{Setting}[section]
\newdefinition{case}{Case}[section]
\newdefinition{framework}{Framework}[section]
\newdefinition{rem}{Remark}[section]
\newdefinition{defn}{Definition}[section]
\newtheorem{cor}{Corollary}[section]
\newtheorem{assum}{Assumption}[section]
\newtheorem{thm}{Theorem}[section]
\newtheorem{prop}{Proposition}[section]
\newtheorem{lem}{Lemma}[section]
\newproof{proof}{Proof}

\numberwithin{equation}{section}
\newcommand{\norm}[1]{\left\Vert#1\right\Vert}
\newcommand{\abs}[1]{\left\vert#1\right\vert}
\newcommand{\set}[1]{\left\{#1\right\}}
\newcommand{\seq}[1]{\left<#1\right>}

\newcommand{\E}[1]{\mathbb{E}\left[#1\right]}

\newcommand{\PP}[1]{\mathbb{P}\left[ #1\right] }

\newcommand{\cov}[1]{\mathrm{Cov}\left( #1\right) }

\begin{document}
	\begin{frontmatter} 
		\title{Error analysis for empirical risk minimization over clipped ReLU networks in solving linear Kolmogorov partial differential equations}

		\author[1]{Jichang Xiao\corref{cor1}}
		\ead{xiaojc19@mails.tsinghua.edu.cn}
		\address[1]{Department of Mathematical Sciences, Tsinghua University,  Beijing 100084, China}

		\author[2]{Xiaoqun Wang}
		\ead{wangxiaoqun@mail.tsinghua.edu.cn}
		
		\address[2]{Department of Mathematical Sciences, Tsinghua University,  Beijing 100084, China}

		\cortext[cor1]{Corresponding author}
		\begin{abstract}
			Deep learning algorithms have been successfully applied to numerically solve linear Kolmogorov partial differential equations~(PDEs). A recent research shows that if the initial functions are bounded, the empirical risk minimization (ERM) over clipped ReLU networks generalizes well for solving the linear Kolmogorov PDE. In this paper, we propose to use a truncation technique to extend the generalization results for polynomially growing initial functions. Specifically, we prove that under an assumption, the sample size required to achieve an generalization error within $\varepsilon$ with a confidence level $\varrho$ grows polynomially in the size of the clipped neural networks and $(\varepsilon^{-1},\varrho^{-1})$, which means that the curse of dimensionality is broken. Moreover, we verify that the required assumptions hold for Black-Scholes PDEs and heat equations which are two important cases of linear Kolmogorov PDEs. For the approximation error, under certain assumptions, we establish approximation results for clipped ReLU neural networks when approximating the solution of Kolmogorov PDEs. Consequently, we establish that the ERM over artificial neural networks indeed overcomes the curse of dimensionality for a larger class of linear Kolmogorov PDEs.
			
		\end{abstract}
		
	\end{frontmatter}
	
	\section{Introduction}
	
	Partial differential equations~(PDEs) have been widely used in modeling problems in physics, finance and engineering. Traditional numerical methods for solving PDEs such as finite differences~\cite{Thomas:2013} and finite elements~\cite{Braess:2007finite} suffer from the curse of dimensionality, which implies that the computational cost grows exponentially as the dimension increases. 
	Recently, many deep learning-based algorithms have been proposed for different classes of PDEs, see \cite{Beck:2018,Yu:2018,Han_2018,raissi2018deep,Raissi:2019physics,Richter2022,Sirignano:2018}. 
	
	In this paper, we focus on the numerical approximation of linear Kolmogorov PDEs, with a special attention to heat equations and the Black-Scholes PDEs. Beck et al.~\cite{Beck:2018} first proposed a deep learning algorithm to approximate the solution over a full hypercube by reformulating the approximation problem into an optimization problem. Berner et al.~\cite{Berner:2020b} further developed the deep learning algorithm to solve parametric Kolmogorov PDEs. Richter et al.~\cite{Richter2022} introduced different formulations of the risk functional to enhance the robustness of the deep learning algorithm. The numerical results in \cite{Beck:2018,Berner:2020b,Richter2022} demonstrated the effectiveness of the deep learning algorithms even in high dimensions. These findings suggest that deep learning algorithms for solving linear Kolmogorov PDEs do not suffer from the curse of dimensionality. 
	
	One theoretical explanation for this phenomenon is that the solutions of Kolmogorov PDEs can be approximated by deep artificial neural networks in which the number of parameters grows polynomially in both the prescribed accuracy $\varepsilon\in (0,1)$ and the dimension $d$, see \cite{Elbr_chter:2021,Grohs:2023,Hutzenthaler_2020,Jentzen:2021}. Besides the analysis of the approximation error for deep artificial neural networks, Berner et al.~\cite{Berner:2020a} proved that under suitable conditions, the generalization error resulted from empirical risk minimization~(ERM) also breaks the curse of dimensionality. From the perspective of learning theory, the generalization results in \cite{Berner:2020a} shows the agnostic probably approximately correct (PAC) learnability of the clipped ReLU neural networks for the quadratic loss function and the region $[u,v]^{d} \times [-D,D]$. 
	
	Since the Hoeffding's inequality is commonly used to establish the PAC-type inequality, the generalization results in \cite{Berner:2020a} are established only for  linear Kolmogorov PDEs with bounded initial functions. There are many linear Kolmogorov PDEs in which the initial functions are in fact unbounded, for example, the Black-Scholes PDEs in the problems of the option pricing of basket call and call on max options. In this paper, we prove that for linear Kolmogorov PDEs with unbounded initial function, the generalization error still does not suffer from the curse of dimensionality if some assumptions are satisfied. 
	
	Under an assumption regarding the tail probability, we extend the generalization results in \cite{Berner:2020a} for a broader class of linear Kolmogorov PDEs, where the initial functions grow polynomially instead of being bounded. We can prove that the required assumption holds for some important equations such as  Black-Scholes PDEs and heat equations. In fact, the extended generalization results hold for the general regression problem under Setting \ref{set:optim}. The truncation technique makes it possible to establish the PAC bound of the generalization error when the empirical loss is unbounded. Specifically, the idea is to truncate both the true loss and empirical loss, then we can apply the Hoeffding's inequality to the truncated parts which are now bounded. By carefully choosing an appropriate truncation threshold, we can establish generalization error bound for ERM over the clipped neural networks. 
	As for the approximation error, we extend the approximation results in \cite{Grohs:2023} and establish approximation properties of clipped neural networks for solutions of linear Kolmogorov PDEs. By summarizing the results of generalization and approximation errors, we demonstrate that, under certain assumptions, the ERM approach indeed overcomes the curse of dimensionality for linear Kolmogorov PDEs. Moreover, We verify that these assumptions hold for specific heat equations and Black-Scholes equations where the initial functions are polynomially growing.
	
	This paper is organized as follows. In Section \ref{sec:back}, we present the generalized supervised regression problem in Setting \ref{set:optim} for the deep learning algorithm in solving the linear Kolmogorov PDEs. We also introduce the background knowledge related to the ERM principle and the error decomposition. In Section~\ref{sec:main}, under certain assumptions, we establish theoretical results of generalization and approximation errors. For two important linear Kolmogorov PDEs, namely the heat equation and Black-Scholes PDEs, we prove that these assumptions hold if the initial functions grows polynomially and can be approximated by ReLU neural networks without the curse of dimensionality. Section \ref{sec:conclusion} concludes the paper.
	
	\section{Background Knowledge}\label{sec:back}
	\subsection{Deep learning and empirical risk minimization}\label{sec:2.1}
	Let $d\in \mathbb{N}$, $T\in(0,\infty)$ and let $f_{d}(x,t)\in C(\mathbb{R}^{d}\times[0,T],\mathbb{R})$ be the solution of the linear Kolmogorv PDE which satisfies for every $x\in \mathbb{R}^{d}$, $t\in (0,T)$ that
	\begin{align}
		\label{eq:target}
		\begin{cases}
			(\nabla_{t}f_{d})(x,t) = \frac{1}{2}\mathrm{Trace}\left( \sigma_{d}(x)[\sigma_{d}(x)]^{*}(\nabla_{x}^{2}f_{d})(x,t)\right) +\seq{\mu_{d}(x),(\nabla_{x}f_{d})(x,t)}_{\mathbb{R}^{d}}, \\
			f_{d}(x,0) = \varphi_{d}(x)
		\end{cases}
	\end{align}
	where $\mu_{d}:\mathbb{R}^{d} \to \mathbb{R}^{d}$ and $\sigma_{d}:\mathbb{R}^{d} \to \mathbb{R}^{d\times d}$ are Lipschitz continuous, $\varphi_{d}:\mathbb{R}^{d} \to \mathbb{R}$ is the initial function, $\left[ \sigma_{d}(x)\right] ^{*}$ is the transpose of the matrix $\sigma_{d}(x)$.
	
	Beck et al.~\cite{Beck:2018} first proposed a novel deep learning algorithm to approximate $f_{d}(x,T)$ on the hypercube $\in[u,v]^{d}$ by artificial neural networks. The key idea is to reformulate the solutions of the Kolmogorov PDEs as the solutions of infinite dimensional optimization problems, see \cite[Corollary 2.4]{Beck:2018} for details. The fundamental theory in \cite{Beck:2018} considers the strong solution $f_{d}(x,t)\in C^{2,1}(\mathbb{R}^{d}\times[0,T],\mathbb{R})$. For the viscosity solution \cite{Hairer}, especially for the solution of Black-Scholes PDEs, the optimization problem is still valid under some regularity conditions, see \cite[Lemma 3.2]{Berner:2020a} and \cite[Theorem A.1]{Berner:2020b}. 
	
	Let $\set{B_{t},\mathcal{F}_{t}\mid 0\leq t\leq T}$ be the standard $d$-dimensional Brownian motion on the filtered probability space $(\Omega,\mathcal{F}, \left(\mathcal{F} \right)_{t\in [0,T]} ,\mathbb{P})$. Let $X_{d}: \Omega \to [u,v]^{d}$ be uniformly distributed and $\mathcal{F}_{0}$-measurable. Let $\set{S_{t},\mathcal{F}_{t}\mid 0\leq t\leq T}$ be the continuous stochastic process satisfying that for every $t\in [0,T]$ it holds $\mathbb{P}$-a.s.
	\begin{equation}\label{eq:sde}
		S_{t} = X_{d} + \int_{0}^{t}\mu_{d}(S_{u})du + \int_{0}^{t}\sigma_{d}(S_{u})dB_{u}.
	\end{equation}
	The objective function of the optimization problem is given by
	$$\mathcal{E}_{d}(f) = \E{\left( f(X_{d})-\varphi_{d}(S_{T})\right) ^{2}}.$$
	In this paper, we assume that the solution of the optimization problem is indeed $f_{d}(x,T),x\in[u,v]^{d}$. It is beyond the scope of this paper to study on the regularity conditions under which the solution $f_{d}(\cdot,T)$ solves the optimization problem, we refer the readers to \cite{Beck:2018,Berner:2020b,Berner:2020a} for details. The next setting generalizes the specific optimization problem into a special regression problem with the hidden output.
	\begin{setting}[Regression problem ]\label{set:optim}
		On the probability space $\left(\Omega,\mathbb{P}\right)$, let $X_{d}: \Omega \to [u,v]^{d}$ be the input and $Y_{d}:\Omega \to \mathbb{R}^{d}$ be the hidden output. Let $\varphi_{d}:\mathbb{R}^{d} \to \mathbb{R}$ satisfy $\E{\varphi_{d}(Y_{d})^{2}}<\infty$. Then for regression problem of data $(X_{d},\varphi_{d}(Y_{d}))$, define the risk functional  $$\mathcal{E}_{d}(f) = \E{\left( f(X_{d})-\varphi_{d}(Y_{d})\right) ^{2}}$$ for every $f \in \mathcal{L}^{2}([u,v]^{d})$. The optimization problem is given by
		\begin{equation}
			\min_{f \in \mathcal{L}^{2}([u,v]^{d})}\mathcal{E}_{d}(f). \nonumber
		\end{equation}
	\end{setting}

	We cannot directly compute the risk $\mathcal{E}$ due to its dependence on the unknown random variable $S_{T}$. In machine learning theory, the ERM principle solves the optimization problem by minimizing the empirical risk over the hypothesis class $\mathcal{H}\subseteq C([u,v]^{d},\mathbb{R})$. 
	
	\begin{setting}[Empirical Risk Minimization]\label{set:ERM}
		Assume Setting \ref{set:optim}, let 
		$\left\lbrace (\bm x_{i},\bm y_{i})\mid 1\leq i<\infty\right\rbrace $ be the independent and identically distributed~(i.i.d.) samples simulated from the population $(X_{d},Y_{d})$. Let $m\in \mathbb{N}$, define the empirical risk $\mathcal{E}_{d,m}$ by
		\begin{equation}
			\mathcal{E}_{d,m} = \frac{1}{m}\sum_{i=1}^{m}\left( f(\bm x_{i})-\varphi_{d}(\bm y_{i})\right) ^{2}. \nonumber
		\end{equation}
		Let $C([u,v]^{d})$ be the Banach space of continuous functions with norm $\norm{\cdot}_{\infty}$ defined by 
		$$\norm{f}_{\infty} = \underset{x\in [u,v]^{d}}{\max}\abs{f(x)}.$$
		Let the hypothesis class $\mathcal{H}\subseteq C([u,v]^{d})$ be compact. The optimal estimator $f_{d,\mathcal{H}}$ on $\mathcal{H}$ is defined by 
		\begin{equation}
			f_{d,\mathcal{H}} = \underset{f\in \mathcal{H}}{\mathrm{argmin}}\ \mathcal{E}_{d}(f), \nonumber
		\end{equation}
		the estimator obtained by ERM is defined by
		\begin{equation}
			f_{d,m,\mathcal{H}} =\underset{f\in \mathcal{H}}{\mathrm{argmin}}\ \mathcal{E}_{d,m}(f). \nonumber
		\end{equation}	
	\end{setting}
	\begin{rem}
		We can directly simulate the uniform distributed $X_d$, but simulating $Y_d = S_T$ is more complex. For important cases of linear kolmogorov PDEs, such as heat equations and Black-Scholes PDEs, the SDE \eqref{eq:sde} can be solved explicitly. In general cases, we usually use the Euler–Maruyama scheme to approximate $S_{T}$ numerically. The existence of $f_{d,\mathcal{H}}$ and $f_{d,m,\mathcal{H}}$ is guaranteed by the compactness of $\mathcal{H}$ (see Proposition \ref{prop:exist} below), the measurability of $f_{d,m,\mathcal{H}}:\Omega \to \mathcal{H}$ is proved in \cite[Section S.1]{Berner:2020a}.
		In classical learning theory \cite{Anthony:1999,Cucker:2002,Koltchinskii:2011,Shalev:2014}, the output $\varphi_{d}(Y_{d})$ is usually assumed to be bounded, i.e., $\sup_{y\in \mathbb{R}^{d}}\abs{\varphi_{d}(y)}<\infty$.
		In the context of nonparametric regression, truncation techniques are applied to conduct error analysis when $\varphi_{d}(Y_{d})$ has finite moments or follows a sub-Gaussian or sub-exponential distribution, see \cite{gyorfi2002distribution,bauer2019deep,schmidt2020nonparametric,jiao2023deep,kohler2021rate} for details. In this paper, we assume that $\varphi_{d}$ is a polynomially growing function and $Y_d$ satisfies certain tail probability condition. 
	
	\end{rem}
	\begin{prop}\label{prop:exist}
		The optimal estimators $f_{d,\mathcal{H}}$ and $f_{d,m,\mathcal{H}}$ in Setting \ref{set:ERM} exist.
	\end{prop}
	\begin{proof}
		For every $d\in \mathbb{N}$, $f_{1},f_{2}\in \mathcal{H}$, we have
		\begin{eqnarray*}
			\abs{\mathcal{E}(f_{1})-\mathcal{E}(f_{2})} 
			&\leq& \E{\abs{f_{1}(X_{d})-f_{2}(X_{d}) }\abs{f_{1}(X_{d})+f_{2}(X_{d})-2\varphi_{d}(Y_{d}) }  }\\
			&\leq&\left(\E{\abs{f_{1}(X_{d})}+\abs{f_{2}(X_{d})}+2\abs{\varphi_{d}(Y_{d})}} \right)\norm{f_{1}-f_{2}}_{\infty},
		\end{eqnarray*}
		by compactness of $\mathcal{H}$, the risk $\mathcal{E}:\mathcal{H}\to \mathbb{R}$ is Lipschitz continuous on $\mathcal{H}$, thus $f_{d,\mathcal{H}}$ exists. For every $\omega \in \Omega$, $f_{d,m,\mathcal{H}}(\omega)$ exists in the same way.
	\end{proof}
	
	In deep learning, a variety of artificial neural networks are designed for different application tasks. Here we take the widely-used feedforward architecture with the ReLU activation function as the hypothesis class $\mathcal{H}$, we call it ReLU neural networks.
	
	\begin{defn}[ReLU neural networks]\label{dfn:ann}
		Let $L\in \mathbb{N}$, $N_{0}, N_{1},\dots,N_{L}\in \mathbb{N}$, let $A_{l}\in \mathbb{R}^{N_{l-1}\times N_{l}}$ and $B_{l}\in \mathbb{R}^{N_{l-1}}$ for $l=1,2,\dots,L$. Denote the collection of $A_{l},B_{l}$ by 
		$$\theta = \left((A_{1},B_{1}),\dots,(A_{L},B_{L}) \right). $$
		The ReLU neural network $\Phi_{\theta}:\mathbb{R}^{N_{0}}\to \mathbb{R}^{N_{L}}$ with parameters $\theta$ is defined by 
		\begin{equation}
			\Phi_{\theta}(x)=W_{L}\circ\rho \circ W_{L-1}\circ\rho\cdots\circ\rho\circ W_{1}(x), \nonumber
		\end{equation}
		where $\left\lbrace W_{l}(x) = A_{l}x+B_{l},1\leq l\leq L\right\rbrace $ are affine transformations and $\rho(x)$ is the ReLU activation function applied component-wisely, i.e.,
		\begin{equation}
			\rho(x_{1},\dots,x_{N}) = \left(\max\left\lbrace x_{1},0 \right\rbrace,\dots,\max\left\lbrace x_{N},0 \right\rbrace \right) . \nonumber
		\end{equation}
		Let $R\in (0,\infty)$, $a = \left( d, N_{1},\dots,N_{L-1},1\right) \in \mathbb{N}^{L+1}$, define the function class $\mathcal{N}_{a,R}$ of ReLU neural networks with architecture $a$ and parameter bound $R$ by
		\begin{equation}
			\mathcal{N}_{a,R} = \set{\Phi_{\theta}|_{[u,v]^{d}}: \norm{\theta}_{\infty}\leq R}, \nonumber
		\end{equation} 
		where $\norm{\theta}_{\infty} = \max_{1\leq l\leq L}\max\left\lbrace \norm{A_{l}}_{\infty},\norm{B_{l}}_{\infty}\right\rbrace$. Define the depth $L(a)$, the width $W(a)$ and the number of parameters $P(a)$ of the neural networks respectively by 
		\begin{eqnarray}
			L(a) = L, \ \ W(a) = \norm{a}_{\infty}, \ \ P(a) = \sum_{l=1}^{L}N_{l}(N_{l-1}+1). \nonumber
		\end{eqnarray}
		
	\end{defn} 
	
	We consider the restriction of neural networks on $[u,v]^{d}$ since we only intend to approximate the solution on the hypercube $[u,v]^{d}$. The function class $\mathcal{N}_{a,R}$ is indeed a compact subspace of $C([u,v]^{d})$, which can be proved with the Lipschitz continuity (\cite[Theorem 2.6]{Berner:2020a}) of the map $\theta \mapsto \Phi_{\theta}$ and the compactness of $\set{\norm{\theta}_{\infty}\leq R}$, here we equivalent $\theta$ to a vector on $\mathbb{R}^{P(a)}$. The number of parameters $P(a)$ and the parameter bound $R$ are two important quantities that measure the complexity of the function class $\mathcal{N}_{a,R}$.

	\subsection{Estimation error of Empirical Risk Minimization}
	Under Setting \ref{set:optim}, 
	We evaluate the estimator $f_{d,m,\mathcal{H}}$ obtained by ERM with the estimation error
	$$\mathbb{E}\left[ \left( f_{d,m,\mathcal{H}}(X_{d})-f_{d}(X_{d},T)\right) ^{2}\right].$$
	
	For the regression problem with quadratic loss function, we have the classical bias-variance decomposition.
	\begin{cor}\label{cor:composition}
		For the estimation error, we have the decomposition
		\begin{eqnarray*}\label{eq:errrorComposition}
			&&\mathbb{E}\left[ \left( f_{d,m,\mathcal{H}}(X_{d})-f_{d}(X_{d},T)\right) ^{2}\right]\\
			&=&\underbrace{\mathcal{E}_{d}(f_{d,m,\mathcal{H}})-\mathcal{E}_{d}(f_{d,\mathcal{H}})}_{\text{generalization error}}+\underbrace{\norm{f_{d,\mathcal{H}}-f_{d}(\cdot,T)}_{L^{2}(P\circ X_{d}^{-1})}^{2}}_{\text{approximation error}}.
		\end{eqnarray*}		
	\end{cor}
	
	We follow \cite{Beck_2022,Berner:2020a} to refer $\mathcal{E}_{d}(f_{d,m,\mathcal{H}})-\mathcal{E}_{d}(f_{d,\mathcal{H}})$ as the generalization error. The classical learning theory aims to establish the PAC inequality for the generalization error, which provides a bound on the sample size $m$ in terms of complexity measures of the hypothesis class $\mathcal{H}$ such that the generalization error is smaller than $\varepsilon$ with confidence greater than $1-\varrho$, i.e.,
	\begin{equation}\label{eq:genetheory}
		\PP{\mathcal{E}_{d}(f_{d,m,\mathcal{H}})-\mathcal{E}_{d}(f_{d,\mathcal{H}})\leq \varepsilon}\geq 1-\varrho, 
	\end{equation}
	see \cite{Anthony:1999,Cucker:2002,Koltchinskii:2011,Shalev:2014} for details. For linear kolmogorov PDEs, Berner et al. \cite{Berner:2020a} prove that the ERM over clipped ReLU neural networks (see Definition \ref{dfn:clipped} below) overcomes the curse of dimensionality. More precisely, they establish polynomial bounds on sample size $m$ in terms of $\left( d,\varepsilon^{-1},\log \left( \varrho^{-1}\right),P(a),R \right) $ to ensure that the inequality \eqref{eq:genetheory} holds for the clipped ReLU neural networks which
	are truncation of the standard ReLU neural networks and uniformly bounded by a constant $D$.
	\begin{defn}[clipped ReLU neural networks]\label{dfn:clipped}
		Let $D\in (0,\infty)$, define function class $\mathcal{N}_{a,R,D}$ of  clipped ReLU neural networks with architecture $a$ and weight bound $R$ by
		\begin{equation}
			\mathcal{N}_{a,R,D} = \set{\mathrm{Clip}_{D}\circ\Phi\mid \Phi \in \mathcal{N}_{a,R}}, \nonumber
		\end{equation} 
		where the clipping function $\mathrm{Clip}_{D}$ is given by
		\begin{equation}
			\mathrm{Clip}_{D}(x)=\max\left\lbrace x,D \right\rbrace \mathrm{sgn}(x). \nonumber
		\end{equation} 
	\end{defn}
	
	The generalization results in \cite{Berner:2020a} hold for linear Kolmogorov PDEs with  bounded initial functions $\varphi_{d}$ satisfying 
	\begin{equation}
		\sup_{d\in\mathcal{N}}\sup_{y\in\mathbb{R}^{d}}\abs{\varphi_{d}(y)} \leq D<\infty. \nonumber
	\end{equation}
	In Section \ref{sec:main}, we extend this results and prove that ERM over clipped neural networks also overcomes the curse of dimensionality when the function $\varphi_{d}$ are unbounded.

	\section{Main Results}\label{sec:main}
	
	\subsection{Generalization error}\label{sec:gene}
	In this section, we establish the generalization error bound for unbounded function $\varphi_{d}$. We first introduce the truncated function as well as the truncated risk and the truncated empirical risk. Throughout this section, we have $\mathcal{H}=\mathcal{N}_{a,R,D}$.
	\begin{defn}
		Let $\bm 1{\left\lbrace \cdot\right\rbrace} $ be the indicator function, for $K\in (0,\infty)$, the truncated function $\varphi_{d}^{(K)}(\cdot)$ is defined by 
		$$\varphi_{d}^{(K)}(x) = \bm 1{\left\lbrace \norm{x}_{\infty}\leq K\right\rbrace}\varphi_{d}(x), $$
		the truncated risk $\mathcal{E}_{d}^{(K)}\left( \cdot\right) $ is defined by
		$$\mathcal{E}_{d}^{(K)}\left( f\right) = \E{\left( f(X_{d})-\varphi_{d}^{(K)}(Y_{d})\right)^{2} },$$
		and the truncated empirical risk $\mathcal{E}_{d,m}^{(K)}\left( \cdot\right) $ is defined by
		$$\mathcal{E}_{d,m}^{(K)}\left( f\right) = \frac{1}{m}\sum_{i=1}^{m}\left[ f(\bm x_{i})-\varphi_{d}^{(K)}(\bm y_{i})\right] ^{2}.$$
	\end{defn}
	
	The next lemma gives the upper bound of the generalization error in terms of the truncated risk and the truncated empirical risk.
	\begin{lem}\label{lem:geneupperbound}
		Assume Setting \ref{set:ERM}. Let $K\in (0,\infty)$, then for every $\omega \in \Omega$ it holds that
		\begin{eqnarray*}
			\mathcal{E}_{d}(f_{d,m,\mathcal{H}}(\omega))-\mathcal{E}_{d}(f_{d,\mathcal{H}})&\leq& 2\sup_{f \in  \mathcal{H}}\abs{\mathcal{E}_{d}(f)-\mathcal{E}_{d,m}(f)} \\
			&\leq& 2(G_{1}+G_{2}+G_{3}),
		\end{eqnarray*}
		where $G_{1},G_{2},G_{3}$ are defined by
		\begin{eqnarray*}
			G_{1} &=& \sup_{f \in \mathcal{H}}\abs{\mathcal{E}_{d}(f)-\mathcal{E}_{d}^{(K)}(f)}, \\
			G_{2} &=& \sup_{f \in \mathcal{H}}\abs{\mathcal{E}_{d}^{(K)}(f)-\mathcal{E}^{(K)}_{d,m}(f)},\\
			G_{3} &=& \sup_{f \in \mathcal{H}}\abs{\mathcal{E}_{d,m}^{(K)}(f)-\mathcal{E}_{d,m}(f)}.
		\end{eqnarray*}
	\end{lem}
	\begin{proof}
		For every $w \in \Omega$, by definition we have 
		$$\mathcal{E}_{d,m}\left( f_{d,m,\mathcal{H}}(\omega)\right)(\omega)\leq \mathcal{E}_{d,m}\left( f_{d,\mathcal{H}}\right)(\omega),$$ thus
		\begin{eqnarray*}
			&&\mathcal{E}_{d}\left( f_{d,m,\mathcal{H}}(\omega)\right)-\mathcal{E}_{d}\left( f_{d,\mathcal{H}}\right) \\
			&\leq& \mathcal{E}_{d}\left( f_{d,m,\mathcal{H}}(\omega)\right)-\mathcal{E}_{d,m}\left( f_{d,m,\mathcal{H}}(\omega)\right)(\omega)+\mathcal{E}_{d,m}\left( f_{d,\mathcal{H}}\right)(\omega)-\mathcal{E}_{d}\left( f_{d,\mathcal{H}}\right) \\
			&\leq&2\sup\limits_{f \in \mathcal{H}}\abs{\mathcal{E}_{d}\left( f\right)-\mathcal{E}_{d,m}\left( f\right)(\omega)} \\ 
			&\leq& 2(G_{1}+G_{2}+G_{3}).
		\end{eqnarray*}
		The last inequality follows from the triangle inequality. $\square$
	\end{proof}
	
	In recent papers \cite{duan2022convergence,jiao2022rate,jiao2023rate,wu2023convergence}, the term $\sup_{f \in  \mathcal{H}}\abs{\mathcal{E}_{d}(f)-\mathcal{E}_{d,m}(f)}$ in Lemma \ref{lem:geneupperbound} is referred to as the statistical error. For other deep learning algorithms in solving PDEs, polynomial bounds on the sample size $m$, in terms of the size of the neutral networks and the error $\varepsilon$, are established to satisfy the following inequality:
	\begin{equation}
		\E{\mathcal{E}_{d}(f_{d,m,\mathcal{H}})-\mathcal{E}_{d}(f_{d,\mathcal{H}})} \leq \varepsilon. \nonumber
	\end{equation}
	In this paper, we are interested in the PAC-type bound \eqref{eq:genetheory}, by Lemma \ref{lem:geneupperbound}, it suffices to derive conditions such that
	\begin{equation}
		\PP{G_{i}\leq \varepsilon} \geq 1-\varrho, \ \ i=1,2,3.  \nonumber
	\end{equation}
	The next assumption is necessary for the subsequent analysis of $G_{i}$. 
	\begin{assum}\label{ass:basic}
		For the output $\varphi_{d}(Y_{d}) $, let $Y_{d} = \left(Y_{d,1},\dots,Y_{d,d} \right) $. Assume 
		\begin{enumerate}
			\item[(i)] Tail probability condition: There exists constant $c_{1}\in (0,\infty)$ and $\lambda_{1}\in (1,\infty)$ such that for every $d\in \mathbb{N}$, $t\in [1,\infty)$ and $i \in  \set{1,2,\dots,d}$ it holds that
			\begin{equation}
				\PP{\abs{Y_{d,i}}\geq t} \leq 2\exp\left\lbrace -c_{1}\left( \log t\right)^{\lambda_{1}} \right\rbrace. \nonumber
			\end{equation}
			\item[(ii)] Polynomial growth condition: There exist constant $c_{2}\in [1,\infty)$ and $\lambda_{2}\in [2,\infty)$ such that for every $d\in \mathbb{N}$ and $y\in \mathbb{R}^{d}$ it holds that
			\begin{equation}
				\abs{\varphi_{d}\left( y\right)} \leq c_{2}\left(1+\norm{y}_{2}^{\lambda_{2}} \right). \nonumber
			\end{equation}
		\end{enumerate}
	\end{assum}
	\begin{rem}
		The polynomial growth condition of the initial function $\varphi_{d}$ typically holds for applications of the linear Kolmogorov PDE, whereas the tail probability condition for $Y_{d}=S_{T}$ often does not. However, for important applications such as the heat equation and the Black-Scholes PDE, the tail probability condition always holds, we prove this statement in Proposition \ref{prop:verify}.
	\end{rem}
	
	The next lemma is a direct consequence of Assumption \ref{ass:basic}. 
	\begin{lem}\label{lem:finitemoment}
		Assume that the output $\varphi_{d}(Y_{d})$ satisfies Assumption \ref{ass:basic}. For every $m\in [0,\infty)$, we have
		\begin{equation}\label{eq:finitemomemt}
			\mu_{m}^{(Y)}\triangleq \sup_{d\in \mathbb{N},i=1,\dots,d}\E{\abs{Y_{d,i}}^{m}}<\infty.
		\end{equation} 
		Moreover, for every $k\in \mathbb{N}^{+}$, we have 
		\begin{align*}
			M_{k,d}\triangleq \E{\abs{\varphi_{d}(Y_{d})}^{k}}
			&\leq c_{2}^{k}2^{k-1}\left( 1+\mu_{k\lambda_{2}}^{(Y)}d^{\frac{k\lambda_{2}}{2}}\right)  \\
			&= O(d^{\frac{k\lambda_{2}}{2}}).
		\end{align*}
		
	\end{lem}
	\begin{proof}
		Apply the expectation formula in terms of tail probability, we have
		\begin{align*}
			\mu_{m}^{(Y)}&= \sup_{d\in \mathbb{N},i=1,\dots,d}\left[ \int_{0}^{\infty}\PP{\abs{Y_{d,i}}^{m}\geq t}dt\right] \\
			&\leq\sup_{d\in \mathbb{N},i=1,\dots,d}\left[1+ \int_{1}^{\infty}\PP{\abs{Y_{d,i}}^{m}\geq t}dt\right]\\
			&\leq1+\sup_{d\in \mathbb{N},i=1,\dots,d}\int_{0}^{\infty}2\exp\left\lbrace -c_{1}\left( \frac{1}{m}\log t\right)^{\lambda_{1}} \right\rbrace dt \\
			&= 1+ 2\int_{1}^{\infty}\exp\left\lbrace -c_{1}\left( \frac{1}{m}\log t\right)^{\lambda_{1}} \right\rbrace dt \\
			&< \infty,
		\end{align*}
		where the second inequality follows from the tail probability condition of $Y_{d}$. From the polynomial growth condition of function $\varphi_{d}(\cdot)$, we have
		\begin{align}
			\E{\abs{\varphi_{d}(Y_{d})}^{k}}&\leq c_{2}^{k}\E{\left(1+\norm{Y_{d}}_{2}^{\lambda_{2}} \right)^{k}}.\\
			& \leq  c_{2}^{k}\E{2^{k-1}\left(1+\norm{Y_{d}}_{2}^{k\lambda_{2}}\right)} \nonumber \\
			&\leq  c_{2}^{k}2^{k-1} + c_{2}^{k}2^{k-1}  \E{d^{\frac{k\lambda_{2}}{2}-1}\sum_{i=1}^{d}\abs{Y_{d,i}}^{k\lambda_{2}}}  \nonumber \\
			& \leq c_{2}^{k}2^{k-1}\left( 1+\mu_{k\lambda_{2}}^{(Y)}d^{\frac{k\lambda_{2}}{2}}\right) \nonumber\\
			&= O(d^{\frac{k\lambda_{2}}{2}}) , \nonumber
		\end{align}
		where the second inequality follows from
		$$(1+x)^{k}\leq 2^{k-1}(1+x^{k}), \ \ x\in [0,\infty)$$ 
		the third inequality follows from the Jensen's inequality of the convex function $f(x) = x^{k\lambda_{2}/2}$ on $[0,\infty)$, and the last equality follows from \eqref{eq:finitemomemt}. $\square$
	\end{proof}
	
	The following proposition provides the condition on truncation diameter $K$ such that $G_{1}\leq \varepsilon$.
	\begin{prop}\label{prop:trun1}
		Under Assumption \ref{ass:basic}. Let $K, D\in [1,\infty)$, $\varepsilon \in (0,1)$. Then there exists a constant $B_{1}>0$ depending only on $c_{1}$, $c_{2}$, $\lambda_{1}$ and $\lambda_{2}$ such that if
		\begin{equation}
			K\geq   B_{1} \left[ \left( D^{2}+d^{\lambda_{2}}\right) \varepsilon^{-1}\sqrt{d} \right]^{2c_{1}^{-1}}  ,\nonumber
		\end{equation}
		then it holds that 
		\begin{equation}\label{eq:prob0}
			G_{1}=\sup\limits_{f \in \mathcal{H}}\abs{\mathcal{E}_{d}^{(K)}\left( f\right)-\mathcal{E}_{d}\left( f\right)} \leq \varepsilon .
		\end{equation}
	\end{prop}
	
	\begin{proof}
		For every $f \in \mathcal{H}=\mathcal{N}_{a,R,D}$, we have
		\begin{align*}
			\abs{\mathcal{E}_{d}^{(K)}\left( f\right)-\mathcal{E}_{d}\left( f\right)}&=\E{\left( f(X_{d})-\varphi_{d}(Y_{d})\right)^{2} \bm 1{\left\lbrace \norm{Y_{d}}_{\infty}> K\right\rbrace}}\\
			&\leq2\E{f(X_{d})^{2} \bm 1{\left\lbrace \norm{Y_{d}}_{\infty}> K\right\rbrace}}+2\E{\varphi_{d}(Y_{d})^{2}\bm 1{\left\lbrace \norm{Y_{d}}_{\infty}> K\right\rbrace}}\\
			&\leq2D^{2}\PP{\norm{Y_{d}}_{\infty}> K}+2\sqrt{\E{\varphi_{d}(Y_{d})^{4}} \PP{\norm{Y_{d}}_{\infty}> K}}\\
			&\leq\left(2D^{2}+ 2\sqrt{M_{4,d} }\right) \sqrt{\PP{\norm{Y_{d}}_{\infty}> K}},
		\end{align*}
		where the second and third inequality follow from the Cauchy–Schwarz inequality and the last inequality follows from $\PP{\norm{Y_{d}}_{\infty}> K}\leq 1$. From the tail probability condition in Assumption~\ref{ass:basic}, we have 
		\begin{eqnarray*}
			\PP{\norm{Y_{d}}_{\infty}> K}
			&\leq& \sum_{i=1}^{d}\PP{\abs{Y_{d,i}}> K} \\
			&\leq& 2d\exp\left\lbrace -c_{1}\left( \log K\right)^{\lambda_{1}}\right\rbrace.
		\end{eqnarray*}
		In conclusion, inequality \eqref{eq:prob0} holds if 
		\begin{equation}
			\varepsilon \geq \left(2D^{2}+ 2\sqrt{M_{4,d} }\right) \sqrt{2d\exp\left\lbrace -c_{1}\left( \log K\right)^{\lambda_{1}}\right\rbrace}, \nonumber
		\end{equation}
		which is equivalent to 
		\begin{equation}
			\log(K)\geq  \left\lbrace 2c_{1}^{-1}\log\left[ \left( 2D^{2}+2\sqrt{M_{4,d}}\right) \varepsilon^{-1}\sqrt{2d} \right] \right\rbrace ^{\lambda_{1}^{-1}}  .\nonumber
		\end{equation}
		By Lemma \ref{lem:finitemoment}, $M_{4,d}=O(d^{2\lambda_{2}})$ is independent of $D$ and $\varepsilon$. Note that 
		$$\left( 2D^{2}+2\sqrt{M_{4,d}}\right) \varepsilon^{-1}\sqrt{2d} \geq 2\sqrt{2} ,$$
		and we have $x^{\lambda_{1}^{-1}} \leq x$ for every $x\in [1,\infty)$. Thus there exists constant $B_{1}>0$ such that 
		\begin{equation}
			\left\lbrace 2c_{1}^{-1}\log\left[ \left( 2D^{2}+2\sqrt{M_{4,d}}\right) \varepsilon^{-1}\sqrt{2d} \right] \right\rbrace ^{\lambda_{1}^{-1}}  \leq 2c_{1}^{-1}\log \left\lbrace B_{1} \left[ \left( D^{2}+d^{\lambda_{2}}\right) \varepsilon^{-1}\sqrt{d} \right]\right\rbrace . \nonumber
		\end{equation}
		The proof is completed. $\square$
	\end{proof}
	
	For the second term $G_{2}=\sup_{f \in \mathcal{H}}\abs{\mathcal{E}_{d}^{(K)}(f)-\mathcal{E}^{(K)}_{d,m}(f)}$, we can use the same techniques in \cite[Theorem 2.9]{Berner:2020a} since the truncated function $\varphi_{d}^{(K)}$ is bounded. The next theorem presents conditions on both $m$ and $K$ to obtain an error smaller than $\varepsilon$ for $G_{2}$ with confidence $1-\varrho$.

	\begin{prop}\label{prop:main1}
		Under Assumption \ref{ass:basic}, let $K,D\in [1,\infty)$, $\varepsilon, \varrho \in (0,1)$, let $a,R$ be the architecture and the parameter bound (see Definition \ref{dfn:ann}). Let $h_{1}:\mathbb{R}^{6}\to \mathbb{R} $ be a function defined by 
		$$h_{1}(x) = 8c_{2}^{5}x_1^2\left(  x_{3}^{\lambda_{2}/2}+1\right)  ^{5}
		\left[\log2 + x_{5} + 2x_{2}x_3^2x_4^2 + 2\log(x_3)x_3^2x_4^2 + 4\log(C_1x_1x_6)x_3^2x_4^2\right], $$
		where constant $C_{1}=64\max\left\lbrace1,\abs{u},\abs{v} \right\rbrace$. Then for sample size $m$ with
		\begin{equation}
			m\geq K^{5\lambda_{2}}
			h_{1}\left(\varepsilon^{-1},\log(R),W(a),L(a),\log(\varrho^{-1}),D \right),\nonumber
		\end{equation}
		it holds that 
		\begin{equation}\label{eq:concluProp3.7}
			\PP{\sup_{f \in \mathcal{H}}\abs{\mathcal{E}_{d}^{(K)}(f)-\mathcal{E}^{(K)}_{d,m}(f)} \leq \varepsilon}\geq 1-\varrho . 
		\end{equation}
	\end{prop}
	\begin{proof}
		By Assumption \ref{ass:basic}, for truncated function $\varphi_{d}^{(K)}$, we have
		$$\norm{\varphi_{d}^{(K)}}_{\mathcal{L}^{\infty}}\leq c_{2}(d^{\lambda_{2}/2}K^{\lambda_{2}}+1)\triangleq B_{d,K}.$$
		
		Since $\varphi_{d}^{(K)}$ is bounded, we can apply the Hoeffding's inequality as in the proof of \cite[Theorem~2.4]{Berner:2020a} to obtain the following inequality 
		\begin{equation}
			\PP{\sup_{f \in \mathcal{H}}\abs{\mathcal{E}_{d}^{(K)}(f)-\mathcal{E}^{(K)}_{d,m}(f)} \leq \varepsilon}\geq 1- 2\cov{\mathcal{H},\frac{\varepsilon}{8(D+B_{d,K})}}\exp\left(-\frac{m\varepsilon^{2}}{8(B_{d,K}^{2}+D^{2})^{2}} \right), \nonumber
		\end{equation}
		where $\cov{\mathcal{H},r}$ denotes the covering number of the compact subspace $\mathcal{H}\subseteq C([u,v]^{d})$ with radius $r$. The inequality \eqref{eq:concluProp3.7} holds if $m$ satisfies
		\begin{equation}\label{eq:samplesize}
			m\geq 8\varepsilon^{-2}(B_{d,K}^{2}+D^{2})^{2}\left[\log\left(2\varrho^{-1} \right)  +\log\cov{\mathcal{N}_{a,R,D},\frac{\varepsilon}{8(D+B_{d,K})}}  \right].
		\end{equation}
		For $\mathcal{H}=\mathcal{N}_{a,R,D}$, Proposition 2.8 in \cite{Berner:2020a} shows that 
		\begin{equation}\label{eq:samplesize2}
			\log\left[ \cov{\mathcal{N}_{a,R,D},r} \right] \leq P(a)\left[ \log\left(\frac{4L(a)^{2}\max\left\lbrace1,\abs{u},\abs{v} \right\rbrace }{r} \right) + L(a)\log\left(R\norm{a}_{\infty} \right)  \right] .
		\end{equation}
		Summarizing \eqref{eq:samplesize} and \eqref{eq:samplesize2}, the inequality \eqref{eq:samplesize} holds if sample size $m$ satisfies 
		$$m\geq H\left( B_{d,K},\varepsilon^{-1},\log\left(R \right),P(a),W(a),L(a),\log\left(\varrho^{-1}\right),D\right), $$ where the function $H:(0,\infty)^{8} \to \mathbb{R}$ is defined by
		$$H\left(x_{1},\dots,x_{8}\right) 
		=8x_2^2(x_{1}^{2}+x_{8}^{2})^{2}\left[\log 2 + x_{7}+ x_{3}x_{4}x_{6}+\log(x_{5})x_4x_{6}+x_{4}\log\left(\frac{1}{2}C_{1}(x_{1}+x_{8})x_{2}x_6^2 \right)  \right].$$
		If $x_4\leq2x_5^2x_6$ and $x_i\geq 1$ for $1\leq i\leq8$, we have
		\begin{equation}\label{eq:pf2}
			\begin{aligned}
			&\quad H(x_{1},\dots,x_{8}) \\
			&\leq 8x_2^2(x_{1}^{2}+x_{8}^{2})^{2}\left[\log 2 + x_{7}+ 2x_3x_5^2x_6^2+2\log(x_{5})x_5^2x_{6}^2+2x_5^2x_6\log\left(C_{1}x_{1}x_{8}x_{2}x_6^2 \right)  \right] \\
			&\leq 8x_2^2(x_{1}^{2}+x_{8}^{2})^{2}\left[\log 2 + x_{7}+ 2x_3x_5^2x_6^2+2\log(x_{5})x_5^2x_{6}^2+4x_1x_5^2x_6^2\log\left(C_{1}x_{2}x_8 \right)  \right]\\
			&\leq 8x_1^5x_2^2(1+x_{8}^{2})^{2}\left[\log 2 + x_{7}+ 2x_3x_5^2x_6^2+2\log(x_{5})x_5^2x_{6}^2+4x_5^2x_6^2\log\left(C_{1}x_{2}x_8 \right)\right].
			\end{aligned}
		\end{equation}
		Note that 
		\begin{equation}\label{eq:pf1}
			\begin{aligned}
				B_{d,K} &= c_{2}(d^{\lambda_{2}/2}K^{\lambda_{2}}+1) \leq c_{2}K^{\lambda_{2}}\left[ W(a)^{\lambda_{2}/2}+1\right]  \\
				P(a) &= \sum_{l=1}^{L}N_{l}(N_{l-1}+1) \leq 2L(a)W(a)^2 .
			\end{aligned}
		\end{equation}
		Summarizing \eqref{eq:pf2} and \eqref{eq:pf1}, we have
		\begin{small}
			\begin{align*}
				&H\left( B_{d,K},\varepsilon^{-1},\log\left(R \right),P(a),W(a),L(a),\log\left(\varrho^{-1}\right),D\right), \\
				\leq& K^{5\lambda_{2}}
				h_{1}\left(\varepsilon^{-1},\log(R),W(a),L(a),\log(\varrho^{-1}),D \right).
			\end{align*}
		\end{small}
		The proof is completed. $\square$
	\end{proof}
	
	For the third term $G_{3}=\sup_{f \in \mathcal{H}}\abs{\mathcal{E}_{d,m}^{(K)}(f)-\mathcal{E}_{d,m}(f)}$, we notice that $G_{3}(\omega)=0$ for any $\omega\in \Omega$ satisfying
	$$\max_{1\leq i \leq m}\norm{\bm y_{i}(\omega)}_{\infty}\leq K,$$
	hence we can prove the following proposition.
	
	\begin{prop}\label{prop:trunprob}
		Under Assumption \ref{ass:basic}. Let $K\in [1,\infty)$, let $S_{K}$ be the set defined by 
		$$S_{K} = \set{\omega \in \Omega\mid \max_{1\leq i \leq m}\norm{\bm y_{i}(\omega)}_{\infty}\leq K} .$$
		Then we have
		\begin{eqnarray*}
			\PP{\sup_{f \in \mathcal{H}}\abs{\mathcal{E}_{d,m}^{(K)}(f)-\mathcal{E}_{d,m}(f)}> \varepsilon}
			&\leq& \PP{\Omega - S_{K}}  \\
			&\leq& 2md\exp\left\lbrace -c_{1}\left( \log K\right)^{\lambda_{1}} \right\rbrace.
		\end{eqnarray*}
	\end{prop}
	
	\begin{proof}
		From the tail probability condition in Assumption \ref{ass:basic}, we have
		\begin{eqnarray*}
			\PP{\Omega - S_{K}} &\leq& \sum_{i=1}^{m}\PP{\norm{\bm y_{i}}_{\infty}\geq K} \\
			&\leq& \sum_{i=1}^{m}\sum_{j=1}^{d}\PP{\abs{y_{i,j}}\geq K}\\
			&\leq& 2md\exp\left\lbrace -c_{1}\left( \log K\right)^{\lambda_{1}}\right\rbrace .
		\end{eqnarray*}
		The first inequality is obvious since $\mathcal{E}_{d,m}^{(K)}(f)=\mathcal{E}_{d,m}(f)$ on $S_{K}$. $\square$
	\end{proof}
	
	Summarizing Propositions \ref{prop:main1}, \ref{prop:trun1} and \ref{prop:trunprob}, we can obtain the generalization error bound for ERM over the clipped neural networks. The conclusion of Theorem \ref{thm:geneerror} below generalizes the result in \cite[Theorem 2.9]{Berner:2020a}. Here we deal with unbounded initial function $\varphi_{d}$ under the assumption \ref{ass:basic}, while \cite[Theorem 2.9]{Berner:2020a} only considers bounded one.
	
	\begin{thm}\label{thm:geneerror}
		Under Assumption \ref{ass:basic}. Let $\varepsilon, \varrho \in (0,1)$, let $\mathcal{H}=\mathcal{N}_{a,R,D}$. Then there exists a polynomial $h: \mathbb{R}^{6}\to \mathbb{R} $ such that for sample size $m$ with
		\begin{equation}
			m\geq 	h\left(\varepsilon^{-1},\log(R), W(a),L(a),\varrho^{-1},D \right)  ,\nonumber
		\end{equation}
		
		it holds that
		$$\PP{	\mathcal{E}_{d}(f_{d,m,\mathcal{H}})-\mathcal{E}_{d}(f_{d,\mathcal{H}})\leq \varepsilon}\geq 1-\varrho.$$
	\end{thm}
	\begin{proof}
		By Lemma \ref{lem:geneupperbound}, we have
		\begin{align*}
			\PP{	\mathcal{E}_{d}(f_{d,m,\mathcal{H}})-\mathcal{E}_{d}(f_{d,\mathcal{H}})> \varepsilon}
			&\leq \PP{2(G_{1}+G_{2}+G_{3})>\varepsilon}\\
			\leq \PP{G_{1}>\frac{\varepsilon}{6}}&+\PP{G_{2}>\frac{\varepsilon}{6}}+\PP{G_{3}>\frac{\varepsilon}{6}},
		\end{align*}
		thus it suffices to guarantee 
		\begin{equation}\label{eq:Gi}
			\PP{G_{i}>\frac{\varepsilon}{6}} \leq \frac{\varrho}{3}, \ \ i=1,2,3.
		\end{equation}
		By Propositions \ref{prop:trun1}, \ref{prop:main1} and \ref{prop:trunprob}, inequality \eqref{eq:Gi} holds if the following conditions hold
		\begin{equation}\label{eq:condition}
			\begin{split}
				K&\geq B_{1} \left[ 6\left( D^{2}+d^{\lambda_{2}}\right) \varepsilon^{-1}\sqrt{d} \right]^{2c_{1}^{-1}} ,\\
				m&\geq K^{5\lambda_{2}}
				h_{1}\left(6\varepsilon^{-1},\log(R),W(a),L(a),\log(3\varrho^{-1}),D \right),\\
				\exp\left\lbrace c_{1}\left( \log K\right)^{\lambda_{1}}\right\rbrace &\geq 6md\varrho^{-1},
			\end{split}
		\end{equation}
		where the constant $B_{1}$ and the function $h_{1}$ are introduced in Propositions \ref{prop:trun1} and \ref{prop:main1}. 
		
		Let $K=m^{1/(6\lambda_{2})}\geq 1$, \eqref{eq:condition} becomes 
		\begin{equation}\label{eq:condition2}
			\begin{split}
				&m \geq B_{1} \left[ 6\left( D^{2}+d^{\lambda_{2}}\right) \varepsilon^{-1}\sqrt{d} \right]^{12\lambda_{2} c_{1}^{-1}} ,\\
				&m\geq  h_{1}^{6\lambda_{2}}\left(6\varepsilon^{-1},\log(R),W(a),L(a),\log(3\varrho^{-1}),D \right)   ,\\
				&\exp\left\lbrace \frac{c_{1}}{(6\lambda_{2})^{\lambda_{1}}}\left( \log m\right)^{\lambda_{1}}-\log m\right\rbrace \geq 6d\varrho^{-1},
			\end{split}
		\end{equation}
		For the last inequality in \eqref{eq:condition2}, Assume that 
		\begin{equation}\label{eq:condition3}
			m \geq 6\exp\left\lbrace (6\lambda_{2})^{\lambda_{1}/(\lambda_{1}-1)}(2^{-1}c_{1})^{1/(1-\lambda_{1})}\right\rbrace d\varrho^{-1}, 
		\end{equation}
		then we have 
		\begin{equation}
			\exp\left\lbrace \frac{c_{1}}{(6\lambda_{2})^{\lambda_{1}}}\left( \log m\right)^{\lambda_{1}}-\log m\right\rbrace 
			\geq  \exp\left\lbrace\log m  \right\rbrace  \geq 6d\varrho^{-1}. \nonumber
		\end{equation}
		Summarizing \eqref{eq:condition2} and \eqref{eq:condition3}, we can prove that the desired polynomial $h$ exists. $\square$
	\end{proof}
	
	Theorem \ref{thm:geneerror} shows that for the ERM principle over clipped ReLU neural networks, if the unbounded output $\varphi_{d}(Y)$ satisfies Assumption \ref{ass:basic}, then the sample size $m$ required to achieve a generalization error small than $\varepsilon$ with confidence  $1-\varrho$ grows polynomially in $\varepsilon^{-1}$, $\varrho^{-1}$ and the complexity measures of networks $(W(a),L(a),R,D)$. This means that the curse of dimensionality is avoided in the generalization error. 
	
	\subsection{Approximation Error}
	From Corollary \ref{cor:composition}, we know that the estimation error can be decomposed into two parts: generalization error and approximation error. In Theorem \ref{thm:geneerror}, we have proved that, under Assumption \ref{ass:basic}, the generalization error part overcomes the curse of dimensionality. The approximation error measures how accurately the target $f_{d}(\cdot,T)$ can be approximated by neural networks in $\mathcal{H}=\mathcal{N}_{a,R,D}$. There are many research results on the approximation properties of neural networks, for examples \cite{Burger:2001error,Grohs:2023,Hornik:1991approximation,Ohn:2019smooth,elbr:2021deep}. For specific Kolmogorov PDEs, research including \cite{Elbr_chter:2021,Grohs:2023,Hutzenthaler_2020,Jentzen:2021,Berner:2020a} demonstrate that the solution can be approximated by deep artificial neural networks in which the number of parameters grows polynomially in both the prescribed accuracy $\varepsilon\in (0,1)$ and the dimension $d$. In this section, we prove that, under Assumption \ref{ass:basic_approx}, the clipped neural networks can also breaks the curse of dimensionality when approximating the solution $f_{d}(\cdot,T)$ of \eqref{eq:target}.
	
	\begin{assum}\label{ass:basic_approx}
		For the Kolmogorov PDE \eqref{eq:target}, assume that 
		\begin{enumerate}
			\item[(i)]The drift $\mu_{d}$ and diffusion $\sigma_{d}$ are affine functions and there exists constant $L$ independent of dimension $d$ such that
			\begin{equation}
				\norm{\mu_{d}(x)}+\norm{\sigma_{d}(x)}_{HS} \leq L(1+\norm{x}), \nonumber
			\end{equation}
			where $\norm{\cdot}_{HS}$ denotes the Hilbert-Schmidt norm on $\mathbb{R}^{d\times d}$.
			\item[(ii)] For every $\varepsilon \in (0,1)$, there exist ReLU neural networks $\Psi_{d,\varepsilon}$ with parameters $\eta_{d,\varepsilon}$ such that for every $x\in\mathbb{R}^{d}$ it holds that
			\begin{equation}\label{eq:ass2}
				\abs{\Psi_{d,\varepsilon}(x)} \leq C(1+\norm{x}^{\lambda}) \ \ \mathrm{and} \ \ \abs{\Psi_{d,\varepsilon}(x)-\varphi_{d}(x)} \leq \varepsilon(1+\norm{x}^{\lambda}), 
			\end{equation}
			where $\lambda\in [2,\infty)$ and $C\in(0,\infty)$ are constants independent of dimension $d$. Moreover, there exist a polynomial $g$ such that
			\begin{equation}\label{eq:psi}
			\max\{\mathcal{P}(\Psi_{d,\varepsilon}),\norm{\eta_{d,\varepsilon}}_\infty\}\leq g(d,\varepsilon^{-1}),
			\end{equation} 
			where $\mathcal{P}(\Psi_{d,\varepsilon})$ denote the number of parameters of $\Psi_{d,\varepsilon}$.
		\end{enumerate}
	\end{assum}

	\begin{prop}\label{prop:extension}
		Let $d\in \mathbb{N}$,  $\varepsilon \in (0,1)$, $T\in (0,\infty)$, $u\in \mathbb{R}$ and $v\in(u,\infty)$. Under Assumption \ref{ass:basic_approx}, there exists a polynomial $q$ such that for every $\varepsilon \in (0,1)$, there exists clipped ReLU neural networks $\Phi_{d,\varepsilon}\in \mathcal{N}_{a,R,D}$ satisfying
		\begin{itemize}
			\item[(\romannumeral1)] $\frac{1}{(v-u)^{d}}\norm{\Phi_{d,\varepsilon}-f_{d}(\cdot,T)}_{\mathcal{L}^{2}([u,v]^{d})}^{2} \leq \varepsilon.$
			\item[(\romannumeral2)] $\max\left\lbrace D,R,P(a) \right\rbrace \leq q(d,\varepsilon^{-1}).$
		\end{itemize}
	\end{prop}
	
	\begin{proof}
		By using the techniques in proof of \cite[Theorem 3.3]{Berner:2020a}, we can extend the results in \cite[Lemma 3.9]{Grohs:2023}. Consequently, there exist neural networks $\Lambda_{d,\varepsilon}$ with parameter $\theta_{d,\varepsilon}$ such that
		\begin{itemize}
			\item[(i)]$$\frac{1}{(v-u)^{d}}\int_{[u,v]^{d}}\abs{f_{d}(x,T)-\Lambda_{d,\varepsilon}(x)}^{2}dx  \leq \varepsilon. \nonumber$$
			\item[(ii)]$\mathcal{P}(\Lambda_{d,\varepsilon})\leq C_2 d^{2\lambda}\varepsilon^{-2}\mathcal{P}(\Psi_{d,\varepsilon})$ and $\norm{\theta_{d,\varepsilon}}_{\infty}\leq C_2 d^{\lambda+3/2}\varepsilon^{-1}\norm{\eta_{d,\varepsilon}}_{\infty},$ where $C_2 \in (0,\infty)$ is a constant.
		\end{itemize}
		
		It suffices to replace the neural networks by the clipped one. From Assumption \ref{ass:basic_approx},
		For every $x\in [u,v]^{d}$, $\varepsilon \in (0,1)$, we have
		\begin{align*}
			\abs{\varphi_{d}(x)}
			&\leq \abs{\Psi_{d,\varepsilon}(x)-\varphi_{d}(x)}+\abs{\Psi_{d,\varepsilon}(x)}\\
			&\leq(C+\varepsilon)(1+\norm{x}^{\lambda}).
		\end{align*} 
		From \cite[Lemma A.3]{Berner:2020a}, we can easily prove that $$\sup_{x\in [u,v]^{d}}\abs{f_{d}(x,T)} = O(d^{\lambda/2})$$. Let $D=\sup_{x\in [u,v]^{d}}\abs{f_{d}(x,T)}$, from \cite[Lemma A.1]{Berner:2020a}, the ReLU neural networks determined by parameter 
		\begin{equation}
			\eta = \left( 
			\left(\begin{bmatrix} 1 \\ -1  \end{bmatrix},\begin{bmatrix} 0 \\ 0  \end{bmatrix} \right),
			\left(\begin{bmatrix} -1&0 \\ 0&-1  \end{bmatrix},\begin{bmatrix} D \\ D  \end{bmatrix} \right),
			\left(\begin{bmatrix} -1&1 \end{bmatrix},\begin{bmatrix} 0&0  \end{bmatrix} \right) \right) . \nonumber
		\end{equation}
		is exactly function $\mathrm{Clip}_{D}(x)$ in Definition \ref{dfn:clipped}. By composition of neural networks (\cite[section 5]{Elbr_chter:2021}), thus there exist
		a network $\Phi_{d,\varepsilon}\in \mathcal{N}_{a,R,D}$ such that
		$$\Phi_{d,\varepsilon} = \mathrm{Clip}_{D}\circ\Lambda_{d,\varepsilon} = \min(\abs{\Lambda_{d,\varepsilon}},D)\mathrm{sgn}(\Lambda_{d,\varepsilon}).$$
		Since $D = \sup_{x\in [u,v]^{d}}\abs{f_{d}(x,T)}$, it is obviously
		\begin{eqnarray*}
			\int_{[u,v]^{d}}\abs{f_{d}(x,T)-\Phi_{d,\varepsilon}(x)}^{2}dx  
			&\leq& \int_{[u,v]^{d}}\abs{f_{d}(x,T)-\Lambda_{d,\varepsilon}(x)}^{2}dx \\
			&\leq& \varepsilon.
		\end{eqnarray*}
		Compositing with the network determined by $\eta$ does not change the magnitude of the number of parameters, but the parameter bound becomes $\max\left\lbrace \norm{\theta_{d,\varepsilon}}_{\infty},D\right\rbrace $, notice $D=O(d^{\lambda/2})$ and \eqref{eq:psi}, the proof is completed. $\square$
	\end{proof}
	
	The Proposition \ref{prop:extension} demonstrate that if we can approximate the initial function $\varphi_{d}$ using neural networks, where the number of parameters $\mathcal{P}(\Psi_{d,\varepsilon})$ and the parameter bound $\norm{\theta_{d,\varepsilon}}_{\infty}$ both grows polynomially in $d$ and $\varepsilon^{-1}$, then we can also approximate the solution $f_{d}(\cdot,T)$ by the clipped ReLU neural networks without the curse of dimensionality. In most applications of Kolmogorov PDEs, the initial function has a simpler closed form, making it easier to verify the Assumption \ref{ass:basic_approx}.

	\subsection{Estimation error}
	In this section, by summarizing the results of generalization and approximation errors, we can conclude that, under some assumptions, the ERM over clipped ReLU neural networks overcomes the curse of dimensionality for the overall estimation error.
	\begin{thm}\label{thm:totalerror}
		Assume Setting \ref{set:ERM}, suppose that Assumptions \ref{ass:basic} and \ref{ass:basic_approx} hold, then there exist polynomials $p$ and $q$ such that for every $d,m\in \mathbb{N}$, $\varepsilon,\varrho \in (0,1)$ with 
		\begin{equation}
			m \geq p(d,\varepsilon^{-1},\varrho^{-1}), \nonumber \\
		\end{equation}
		there exist clipped ReLU neural networks $\mathcal{H} = N_{a,R,D}$ such that
		\begin{equation}
			\PP{\frac{1}{(v-u)^{d}}\norm{f_{d,m,\mathcal{H}}-f_{d}(\cdot,T)}_{\mathcal{L}^{2}([u,v]^{d})}^{2} \leq \varepsilon}\geq 1-\varrho , \nonumber
		\end{equation}
		and $\max\left\lbrace P(a),R,D\right\rbrace \leq q(d,\varepsilon^{-1})$.
	\end{thm}
	\begin{proof}
		This is a direct consequence of Corollary \ref{cor:composition}, Theorem \ref{thm:geneerror} and Proposition \ref{prop:extension}. 
	\end{proof}

	\subsection{Important cases of linear Kolmogorov PDEs with unbounded initials}\label{sec:example}
	In this section, we introduce two important linear Kolmogorov PDEs: heat equations and Black-Scholes PDEs in which the initial functions are unbounded. 
	\begin{case}[Heat equations] \label{set:heat}
		Let $d\in \mathbb{N}$, denote the $d$-dimensional identity matrix by $I_{d}$. Assume that the drift function $\mu_{d}$ and the diffusion function $\sigma_d$ satisfy
		\begin{equation}
			\mu_{d}(x) = 0, \ \  \sigma_{d}(x) = \sqrt{2} I_{d},
		\end{equation}
		then the linear Kolmogorov PDE \eqref{eq:target} becomes the heat equation defined by
		\begin{equation}
			(\nabla_{t}f_{d})(t,x) = (\Delta_{x}f_{d})(t,x) . \nonumber
		\end{equation}
	\end{case}
	
	For the heat equation on $\mathbb{R}^{d}$, if the initial function $\varphi_{d}$ is globally bounded and continuous, then the solution $f_{d}$ exists and can be expressed as 
	\begin{equation}\label{eq:evans}
		f_{d}(x,t) = \frac{1}{\left(4\pi t\right)^{n/2} }\int_{\mathbb{R}^{d}}e^{-\frac{\norm{x-y}_{2}^{2}}{4t}}\varphi_{d}(y)dy,
	\end{equation}
	moreover, we have $f_{d}\in C^{\infty}(\mathbb{R}^{d}\times (0,\infty))$, see \cite[Chapter 2.3]{evans2010partial} for a rigorous proof. The theoretical results of the generalization error in \cite{Berner:2020a} only hold for heat equations with bounded initial functions. However, for unbounded initial functions such as polynomials, the solution may also exists. In this paper, we pay more attention to the following example.
	\begin{example}[Heat equations with polynomial initial functions]\label{example:heatpoly}
		Assume Setting \ref{set:heat}, let $d,k\in \mathbb{N}$, let $\left\lbrace c_{d,i},1\leq i \leq d\right\rbrace  \subseteq \mathbb{R}$ satisfy that $\sup_{d\in \mathbb{N},i=1,\dots,d}\abs{c_{d,i}}<\infty$. The initial function $\varphi_{d}$ is given by
		$$\varphi_{d}(x)=\sum_{i=1}^{d}c_{d,i}x_{i}^{k}.$$
	\end{example}
	
	For the equations in Example \ref{example:heatpoly}, the initial function is obviously unbounded, which means that the generalization error bound in \cite{Berner:2020a} does not hold. However, we can apply Theorem~\ref{thm:geneerror} to establish a similar error bound when the initial functions are unbounded. We will verify Assumption \ref{ass:basic} in Proposition \ref{prop:verify} below.
	
	Next, we consider the Black-Scholes PDE which governs the price evolution of European options under the Black–Scholes model.
	\begin{case}[Black-Scholes partial differential equation]\label{set:BSmodel}
		Let $d\in \mathbb{N}$, in accordance with reality, assume $ [u,v] \subseteq (0,\infty)$. Let $(\alpha_{d,i})_{i=1}^{d}$, $(\beta_{d,i})_{i=1}^{d}\subseteq \mathbb{R}$ satisfy 
		$$\sup_{d\in\mathbb{N},1\leq i \leq d}\left(\abs{\alpha_{d,i}}+\abs{\beta_{d,i}} \right)<\infty,$$ 
		let $\Sigma_{d} \in \mathbb{R}^{d\times d}$ satisfy $\norm{\Sigma_{i,d}}_{2}=1$, $i=1,2,\dots,d$, where $\Sigma_{i,d}$ is the $i$-th row of $\Sigma_{d}$. Assume the drift function $\mu_{d}$ and the diffusion function $\sigma_d$ satisfy
		\begin{equation}
			\mu_{d}(x) = \left(\alpha_{d,1}x_{1},\dots,\alpha_{d,d}x_{d} \right), \ \ \sigma_{d}(x) = \mathrm{diag}\left(\beta_{d,1}x_{1},\dots,\beta_{d,d}x_{d} \right)\Sigma_{d}. \nonumber
		\end{equation}
	\end{case}
	
	For different choice of initial functions $\varphi_{d}(\cdot)$, Black-Scholes PDEs model the pricing problem for different types of European options. We take basket call options and call on max options in \cite{Grohs:2023} for examples since the corresponding initial functions are unbounded.
	
	\begin{example}[Basket call options]\label{example:basket}Assume Setting \ref{set:BSmodel}, let $\left( c_{d,i}\right)_{i=1}^{d}\subseteq [0,1] $ satisfy $\sum_{i=1}^{d}c_{d,i}=1$, let $K_{d}\in (0,\infty)$ satisfy $\sup_{d\in \mathbb{N}}K_{d} <\infty$. The initial function $\varphi_{d}$ is given by
		\begin{equation} \label{eq:basket}
			\varphi_{d}(x) = \max\left\lbrace \sum_{i=1}^{d}c_{d,i}x_{i}-K_{d},0 \right\rbrace.
		\end{equation}
	\end{example}
	
	\begin{example}[Call on max options]\label{example:callonmax}Assume Setting \ref{set:BSmodel}, let $\left( c_{d,i}\right)_{i=1}^{d}\subseteq [0,\infty) $ satisfying $\sup_{d\in\mathbb{N}, 1\leq i \leq d}c_{d,i}<\infty$, let $K_{d}\in (0,\infty)$ satisfying $\sup_{d\in \mathbb{N}}K_{d} <\infty$. The initial function $\varphi_{d}$ is given by
		\begin{equation} \label{eq:callonmax}
			\varphi_{d}(x) = \max\left\lbrace \max\left\lbrace c_{d,1}x_{1},c_{d,2}x_{2},\cdots,c_{d,d}x_{d}\right\rbrace -K_{d},0 \right\rbrace. 
		\end{equation}
	\end{example}
	
	In the next proposition, for the two important cases, namely the heat equation and the Black-Scholes PDE, we prove that the required Assumption \ref{ass:basic} holds when the initial functions grow polynomially.
	\begin{prop}\label{prop:verify}
		For the heat equation and the Black-Scholes PDE (see Case \ref{set:heat} and \ref{set:BSmodel}), the tail probability condition in Assumption \ref{ass:basic} holds with $\lambda_{1} = 2$. For initial functions of Examples \ref{example:heatpoly}, \ref{example:basket}, \ref{example:callonmax}, the polynomial growth condition in Assumption \ref{ass:basic} holds.
	\end{prop}  
	\begin{proof}
		For the heat equation, we can solve the stochastic differential equation \eqref{eq:sde} explicitly
		$$Y_{d}=S_{T}=X_{d} + \sqrt{2}B_{T}.$$
		Note that $X_{d,i}\in [u,v]$ is bounded and $B_{T,i}$ is normally distributed random variable $N(0,\sqrt{T})$. From the concentration theory \cite[Section 3.5]{vershynin_2018}, $X_{d,i}$ and $B_{T,i}$ are independent sub-Gaussian random variables, hence $Y_{d,i}=X_{d,i}+\sqrt{2}B_{T,i}$ is also sub-Gaussian random variable. By sub-Gaussian properties \cite[Proposition 2.5.2]{vershynin_2018}, there exists a parameter $K$ depending only on on $u,v,T$ such that
		\begin{eqnarray*}
			\PP{\abs{Y_{d,i}}\geq t}
			&\leq& 2\exp\left\lbrace -\frac{t^{2}}{K} \right\rbrace \\
			&\leq& 2\exp\left\lbrace -\frac{\left( \log t\right) ^{2}}{K} \right\rbrace,
		\end{eqnarray*}
		where the second inequality follows from $ 0\leq \log t\leq t$ for $t\in [1,\infty)$, hence the tail probability condition is satisfied. 
		
		For the Black-Scholes PDE, we can also solve the SDE explicitly
		\begin{eqnarray}
			Y_{d,i}=S_{T,i}= X_{d,i}\exp\left\lbrace \left( \alpha_{d,i}-\frac{\norm{\beta_{d,i}\Sigma_{d,i}}^{2}_{2}}{2}\right)T+\beta_{d,i}\left\langle \Sigma_{d,i},B_{T}\right\rangle  \right\rbrace , \nonumber 
		\end{eqnarray}
		where $\Sigma_{d,i}$ is the $i$-th row of matrix $\Sigma_{d}$ and $\left\langle \cdot,\cdot \right\rangle$ is the inner product in $\mathbb{R}^{d}$. For every $d\in \mathbb{N}$, $i\in \set{1,2,\dots,d}$, $t\in [1,\infty)$, we have
		\begin{eqnarray*}
			\PP{\abs{Y_{d,i}}\geq t} &=&\PP{\log\abs{Y_{d,i}}\geq \log t}\\
			&=&\PP{\log\abs{X_{d,i}}+\left( \alpha_{d,i}-\frac{\norm{\beta_{d,i}\Sigma_{d,i}}^{2}_{2}}{2}\right)T+\beta_{d,i}\left\langle \Sigma_{d,i},B_{T} \right\rangle \geq \log t}. 
		\end{eqnarray*}
		Note that in Setting \ref{set:BSmodel}, we assume that 
		$$0<u<v,\ \  \sup_{d\in\mathbb{N},1\leq i \leq d}\left(\abs{\alpha_{d,i}}+\abs{\beta_{d,i}} \right)<\infty,$$ 
		and $\norm{\Sigma_{d,i}}=1$ for every $i\in \left\lbrace1,2,\dots,d \right\rbrace $. Hence the random variable 
		$$\log\abs{X_{d,i}}+\left( \alpha_{d,i}-\frac{\norm{\beta_{d,i}\Sigma_{d,i}}_{2}^{2}}{2}\right)T$$ is bounded and the upper bound is independent of $d$. Since $B_{T}\sim N(0,TI_{d})$, the random variable $\beta_{d,i}\left\langle \Sigma_{d,i},B_{T} \right\rangle$ is normally distributed with mean 0 and variance $T\beta_{d,i}^{2}$. In summary, the random variable $\log \abs{Y_{d,i}}$ is the sum of two independent sub-Gaussian random variables, hence there exists a parameter $K$ independent of $d$ such that 
		\begin{eqnarray*}
			\PP{\abs{Y_{d,i}}\geq t} 
			\leq 2\exp\left\lbrace -\frac{\left( \log t\right) ^{2}}{K} \right\rbrace.
		\end{eqnarray*}
		It is obvious that the initial functions in Examples \ref{example:heatpoly}, \ref{example:basket}, \ref{example:callonmax} satisfy the polynomial growth condition, the proof is completed.
	\end{proof}
	
	Proposition \ref{prop:verify} actually indicates that the Assumption \ref{ass:basic} holds if the initial functions of the heat equation and the Black-Scholes PDE satisfy the polynomial growth condition, which may be obviously ture in most cases. Hence the generalization result in Theorem \ref{thm:geneerror} holds for most deep learning problems arising in solving heat equations and Black-Scholes PDEs, especially Examples \ref{example:heatpoly}, \ref{example:basket}, \ref{example:callonmax}. Moreover, we can prove that the overall error also breaks the curse of dimensionality.
	
	\begin{prop}\label{thm:heat_bs}
		Assume Setting \ref{set:ERM}, for Examples \ref{example:heatpoly}, \ref{example:basket}, \ref{example:callonmax}, there exist polynomials $p$ and $q$ such that for every $d,m\in \mathbb{N}$, $\varepsilon,\varrho \in (0,1)$ with 
		\begin{equation}
			m \geq p(d,\varepsilon^{-1},\varrho^{-1}), \nonumber \\
		\end{equation}
		there exist clipped ReLU neural networks $\mathcal{H} = N_{a,R,D}$ such that
		\begin{equation}
			\PP{\frac{1}{(v-u)^{d}}\norm{f_{d,m,\mathcal{H}}-f_{d}(\cdot,T)}_{\mathcal{L}^{2}([u,v]^{d})}^{2} \leq \varepsilon}\geq 1-\varrho , \nonumber
		\end{equation}
		and $\max\left\lbrace P(a),R,D\right\rbrace \leq q(d,\varepsilon^{-1})$.
	\end{prop}
	\begin{proof}
		From Theorem \ref{thm:totalerror} and Proposition \ref{prop:verify}, we only need to verify the Assumption \ref{ass:basic_approx} for Examples \ref{example:basket} and \ref{example:callonmax}. Obviously, condition $(i)$ in Assumption \ref{ass:basic_approx} holds for Black-Scholes PDEs. For condition $(ii)$, from \cite[Lemma 4.6]{Grohs:2023}, there exists a neural network $\Psi_{d}$ with parameter $\eta_{d}$ such that
		\begin{eqnarray*}
			\Psi_{d}(x)=\max\left\lbrace \sum_{i=1}^{d}c_{d,i}x_{i}-K_{d},0 \right\rbrace,
		\end{eqnarray*}
		$\mathcal{P}(\Psi_{d})\leq 4d$ and $\norm{\eta_{d}}_{\infty}\leq\sup_{d\in \mathbb{N}}\left(K_{d}+c_{d,i} \right) <\infty.$ Network $\Psi_{d}$ satisfies \eqref{eq:ass2} with $\lambda=2$, thus the initial function in Example \ref{example:basket} satisfies the condition (ii). By \cite[Lemma 4.6]{Grohs:2023}, there also exists a neural network $\Psi_{d}$ with parameter $\eta_{d}$ such that
		\begin{eqnarray*}
			\Psi_{d}(x)= \max\left\lbrace \max\left\lbrace c_{d,1}x_{1},c_{d,2}x_{2},\cdots,c_{d,d}x_{d}\right\rbrace -K_{d},0 \right\rbrace,
		\end{eqnarray*}
		$\mathcal{P}(\Psi_{d})\leq 6d^{3} $ and $\norm{\eta_{d}}_{\infty}\leq\sup_{d\in \mathbb{N},i=1,\dots,d}\left(K_{d}+c_{d,i} \right) <\infty$. In conclusion, Assumption \ref{ass:basic_approx} holds for Examples \ref{example:basket} and \ref{example:callonmax}. For Example \ref{example:heatpoly}, the endpoint solution $f_{d}(\cdot,T)$ can be written as 
		\begin{eqnarray*}
			f_{d}(x,T) &=& \sum_{i=1}^{d}\E{c_{d,i}\left(x_{i}+\sqrt{2T}Z_{i} \right)^{k} }\\
			&=&\sum_{i=1}^{d}P_{d,i}(x_{i}).
		\end{eqnarray*} 
		For every $d\in \mathbb{N}$, $1\leq i \leq d$, the degree of polynomial $P_{d,i}$ is $k$ and the coefficients are uniformly bounded.
		Based on approximation results of ReLU neural networks (\cite[Proposition \uppercase\expandafter{\romannumeral3}.5]{elbr:2021deep}, \cite[Theorem 6.5]{Elbr_chter:2021}), there exists a neural network $\Phi_{\varepsilon}:\mathbb{R}\to \mathbb{R}$ with $O(\log\left(\varepsilon^{-1} \right) )$ layers, $O(1)$ width and $O(1)$ parameters bound satisfying
		\begin{equation}
			\norm{\Phi_{\varepsilon}(x)-P_{d,i}(x)}_{\mathcal{L}^{\infty}([u,v])}\leq \varepsilon. \nonumber
		\end{equation}  
		By basic calculus of neural networks (\cite[section 5]{Elbr_chter:2021}, \cite[Section 5]{Jentzen:2021}), there exists ReLU neural networks $\Phi_{d,\varepsilon}$ with at most $O(d^{2}\left( d+\log\left(d\varepsilon^{-1}\right)  \right) )$ number of parameters and $O(1)$ parameter bound satisfying 
		\begin{equation}
			\norm{\Phi_{d,\varepsilon}-\sum_{i=1}^{d}P_{d,i}(x_{i})}_{\mathcal{L}^{\infty}([u,v]^{d})}\leq \varepsilon. \nonumber
		\end{equation}
		Using the same technique as in the proof of Proposition \ref{prop:extension}, we find that the required clipped ReLU neural networks in Proposition \ref{prop:extension} exist. The proof is completed.
	\end{proof}

	\section{Conclusion}\label{sec:conclusion}
	In this paper, we establish two assumptions under which the ERM over clipped ReLU neural networks can overcome the curse of dimensionality for numerical approximation of linear Kolmogorov PDEs when the initial functions are polynomially growing. For Black-Scholes PDEs and heat equations, we prove that the two assumptions are indeed satisfied if the initial function can be approximated by ReLU neural networks without the curse of dimensionality, see Theorem \ref{thm:totalerror}. For specific linear Kolmogorov PDEs such as examples \ref{example:heatpoly}, \ref{example:basket}, \ref{example:callonmax}, we prove that the deep learning algorithms based on the ERM over clipped ReLU neural networks are capable of avoiding of the curse of dimensionality, see Proposition \ref{thm:heat_bs}. 
	
	Classical learning theory usually employs the Hoeffding's inequality to establish the PAC bound for the generalization error. However, when the empirical loss is unbounded, this technique does not work. In this paper, we use the truncation trick to extend the generalization results in \cite{Berner:2020a} to the cases involving unbounded functions $\varphi_{d}$. Moreover, we prove that, for regression problem under Setting \ref{set:optim}, the sample complexity $m$ required to achieve an generalization error within $\varepsilon$ with a confidence $\varrho^{-1}$ grows polynomially both in size of the clipped neural networks and $(\varepsilon^{-1},\varrho^{-1})$, which means that the curse of dimensionality is broken for the generalization part. Furthermore, if the target functions can be approximated by the neural networks without the curse of dimensionality, we can conclude that the overall error still overcomes the curse of dimensionality. 
	
	\section*{Acknowledgments}
	This work is supported by the National Natural Science Foundation of China through grant 72071119.


\begin{thebibliography}{00}
		
		\bibitem{Anthony:1999}
		{\sc M.~Anthony and P.~L. Bartlett}, {\em Neural Network Learning: Theoretical
			Foundations}, Cambridge University Press, 1999.
		
		\bibitem{Beck:2018}
		{\sc C.~Beck, S.~Becker, P.~Grohs, N.~Jaafari, and A.~Jentzen}, {\em Solving
			the {Kolmogorov} {PDE} by means of deep learning}, 2021,
		\url{https://arxiv.org/abs/1806.00421}.
		
		\bibitem{Beck_2022}
		{\sc C.~Beck, A.~Jentzen, and B.~Kuckuck}, {\em Full error analysis for the
			training of deep neural networks}, Infin. Dimens. Anal. Quantum Probab.
		Relat. Top., 25 (2022), \url{https://doi.org/10.1142/s021902572150020x}.
		
		\bibitem{Berner:2020b}
		{\sc J.~Berner, M.~Dablander, and P.~Grohs}, {\em Numerically solving
			parametric families of high-dimensional {Kolmogorov} partial differential
			equations via deep learning}, Advances in Neural Information Processing
		Systems, 33 (2020), pp.~16615--16627.
		
		\bibitem{Berner:2020a}
		{\sc J.~Berner, P.~Grohs, and A.~Jentzen}, {\em Analysis of the generalization
			error: Empirical risk minimization over deep artificial neural networks
			overcomes the curse of dimensionality in the numerical approximation of
			{Black}--{Scholes} partial differential equations}, SIAM J. Math. Data Sci.,
		2 (2020), pp.~631--657, \url{https://doi.org/10.1137/19M125649X}.
		
		\bibitem{Braess:2007finite}
		{\sc D.~Braess}, {\em Finite Elements: Theory, Fast Solvers, and Applications
			in Solid Mechanics}, Cambridge University Press, 3~ed., 2007,
		\url{https://doi.org/10.1017/CBO9780511618635}.
		
		\bibitem{Burger:2001error}
		{\sc M.~Burger and A.~Neubauer}, {\em Error bounds for approximation with
			neural networks}, J. Approx. Theory, 112 (2001), pp.~235--250,
		\url{https://doi.org/10.1006/jath.2001.3613}.
		
		\bibitem{Cucker:2002}
		{\sc F.~Cucker and S.~Smale}, {\em On the mathematical foundations of
			learning}, Bull. Amer. Math. Soc. (N.S.), 39 (2002),
		pp.~1--49.
		
		\bibitem{Yu:2018}
		{\sc W.~E and B.~Yu}, {\em The deep {Ritz} method: a deep learning-based
			numerical algorithm for solving variational problems}, Commun. Math. Stat., 6
		(2018), pp.~1--12, \url{https://doi.org/10.1007/s40304-018-0127-z}.
		
		\bibitem{Elbr_chter:2021}
		{\sc D.~Elbrächter, P.~Grohs, A.~Jentzen, and C.~Schwab}, {\em {DNN}
			expression rate analysis of high-dimensional {PDEs}: Application to option
			pricing}, Constr. Approx., 55 (2021), pp.~3--71,
		\url{https://doi.org/10.1007/s00365-021-09541-6}.
		
		\bibitem{elbr:2021deep}
		{\sc D.~Elbrächter, D.~Perekrestenko, P.~Grohs, and H.~Bölcskei}, {\em Deep
			neural network approximation theory}, 2021,
		\url{https://arxiv.org/abs/1901.02220}.
		
		\bibitem{evans2010partial}
		{\sc L.~Evans}, {\em Partial Differential Equations}, Graduate studies in
		mathematics, American Mathematical Society, 2010,
		\url{https://books.google.com/books?id=Xnu0o_EJrCQC}.
		
		\bibitem{Grohs:2023}
		{\sc P.~Grohs, F.~Hornung, A.~Jentzen, and P.~von Wurstemberger}, {\em A proof
			that artificial neural networks overcome the curse of dimensionality in the
			numerical approximation of {Black}{\textendash}{Scholes} partial differential
			equations}, Mem. Amer. Math. Soc., 284 (2023),
		\url{https://doi.org/10.1090/memo/1410}.
		
		\bibitem{Hairer}
		{\sc M.~Hairer, M.~Hutzenthaler, and A.~Jentzen}, {\em Loss of regularity for
			{KOLMOGOROV} equations}, Ann. Probab., 43 (2015), pp.~468--527,
		\url{http://www.jstor.org/stable/24519152} (accessed 2023-10-06).
		
		\bibitem{Han_2018}
		{\sc J.~Han, A.~Jentzen, and W.~E}, {\em Solving high-dimensional partial
			differential equations using deep learning}, Proc. Natl. Acad. Sci. USA, 115
		(2018), pp.~8505--8510, \url{https://doi.org/10.1073/pnas.1718942115}.
		
		\bibitem{Hornik:1991approximation}
		{\sc K.~Hornik}, {\em Approximation capabilities of multilayer feedforward
			networks}, Neural Networks, 4 (1991), pp.~251--257,
		\url{https://doi.org/10.1016/0893-6080(91)90009-T}.
		
		\bibitem{Hutzenthaler_2020}
		{\sc M.~Hutzenthaler, A.~Jentzen, T.~Kruse, and T.~A. Nguyen}, {\em A proof
			that rectified deep neural networks overcome the curse of dimensionality in
			the numerical approximation of semilinear heat equations}, Partial Differ.
		Equ. Appl., 1 (2020), \url{https://doi.org/10.1007/s42985-019-0006-9}.
		
		\bibitem{Jentzen:2021}
		{\sc A.~Jentzen, D.~Salimova, and T.~Welti}, {\em A proof that deep artificial
			neural networks overcome the curse of dimensionality in the numerical
			approximation of {Kolmogorov} partial differential equations with constant
			diffusion and nonlinear drift coefficients}, Commun. Math. Sci., 19 (2021),
		pp.~1167--1205, \url{https://doi.org/10.4310/cms.2021.v19.n5.a1}.
		
		\bibitem{Koltchinskii:2011}
		{\sc V.~Koltchinskii}, {\em Oracle Inequalities in Empirical Risk Minimization
			and Sparse Recovery Problems: {\'E}cole D’{\'E}t{\'e} de Probabilit{\'e}s
			de Saint-Flour XXXVIII-2008}, vol.~2033, Springer Science \& Business Media,
		2011, \url{https://doi.org/10.1007/978-3-642-22147-7}.
		
		\bibitem{Ohn:2019smooth}
		{\sc I.~Ohn and Y.~Kim}, {\em Smooth function approximation by deep neural
			networks with general activation functions}, Entropy, 21 (2019),
		\url{https://doi.org/10.3390/e21070627}.
		
		\bibitem{raissi2018deep}
		{\sc M.~Raissi}, {\em Deep hidden physics models: Deep learning of nonlinear
			partial differential equations}, 2018,
		\url{https://arxiv.org/abs/1801.06637}.
		
		\bibitem{Raissi:2019physics}
		{\sc M.~Raissi, P.~Perdikaris, and G.~Karniadakis}, {\em Physics-informed
			neural networks: A deep learning framework for solving forward and inverse
			problems involving nonlinear partial differential equations}, J. Comput.
		Phys., 378 (2019), pp.~686--707,
		\url{https://doi.org/10.1016/j.jcp.2018.10.045}.
		
		\bibitem{Richter2022}
		{\sc L.~Richter and J.~Berner}, {\em Robust {SDE}-based variational
			formulations for solving linear {PDEs} via deep learning}, in INTERNATIONAL
		CONFERENCE ON MACHINE LEARNING, VOL 162, Proceedings of Machine Learning
		Research, 2022.
		\newblock 38th International Conference on Machine Learning (ICML), Baltimore,
		MD, JUL 17-23, 2022.
		
		\bibitem{Shalev:2014}
		{\sc S.~Shalev-Shwartz and S.~Ben-David}, {\em Understanding Machine Learning:
			From Theory to Algorithms}, Cambridge University Press, 2014.
		
		\bibitem{Sirignano:2018}
		{\sc J.~Sirignano and K.~Spiliopoulos}, {\em {DGM}: A deep learning algorithm
			for solving partial differential equations}, J. Comput. Phys., 375 (2018),
		pp.~1339--1364, \url{https://doi.org/10.1016/j.jcp.2018.08.029}.
		
		\bibitem{Thomas:2013}
		{\sc J.~i. Thomas}, {\em Numerical Partial Differential Equations: Finite
			Difference Methods}, vol.~22, Springer Science \& Business Media, 2013,
		\url{https://doi.org/10.1007/978-1-4899-7278-1}.
		
		\bibitem{vershynin_2018}
		{\sc R.~Vershynin}, {\em High-Dimensional Probability: An Introduction with
			Applications in Data Science}, Cambridge Series in Statistical and
		Probabilistic Mathematics, Cambridge University Press, 2018,
		\url{https://doi.org/10.1017/9781108231596}.
		
		\bibitem{duan2022convergence}
		Chenguang Duan, Yuling Jiao, Yanming Lai, Dingwei Li, Jerry~Zhijian Yang,
		et~al.
		\newblock Convergence rate analysis for deep ritz method.
		\newblock {\em Communications in Computational Physics}, 31(4):1020--1048,
		2022.
		
		\bibitem{jiao2022rate}
		Yuling Jiao, Yanming Lai, Dingwei Li, Xiliang Lu, Fengru Wang, Jerry~Zhijian
		Yang, et~al.
		\newblock A rate of convergence of physics informed neural networks for the
		linear second order elliptic pdes.
		\newblock {\em Communications in Computational Physics}, 31(4):1272--1295,
		2022.
		
		\bibitem{jiao2023rate}
		Yuling Jiao, Jerry~Zhijian Yang, Junyu Zhou, et~al.
		\newblock A rate of convergence of weak adversarial neural networks for the
		second order parabolic pdes.
		\newblock {\em Communications in Computational Physics}, 34(3):813--836, 2023.
		
		\bibitem{wu2023convergence}
		Sidi Wu, Aiqing Zhu, Benzhuo Lu, et~al.
		\newblock Convergence of physics-informed neural networks applied to linear
		second-order elliptic interface problems.
		\newblock {\em Communications in Computational Physics}, 33(2):596--627, 2023.
		
			\bibitem{bauer2019deep}
		{\sc B.~BAUER and M.~KOHLER}, {\em On deep learning as a remedy for the curse of dimensionality in
			nonparametric regression}, The Annals of Statistics, 47(4):2261--2285, 2019.
		
		\bibitem{schmidt2020nonparametric}
		{\sc S.~HIEBER and JOHANNES}, {\em nonparametric regression using deep neural networks with relu activation function}, The Annals of Statistics, 48(4):1875--1897, 2020.
		
		\bibitem{jiao2023deep}
		{\sc Jiao, Yuling and Shen, Guohao and Lin, Yuanyuan and Huang, Jian}, {\em Deep nonparametric regression on approximate manifolds: Nonasymptotic error bounds with polynomial prefactors}, The Annals of Statistics, 51(2):691--716, 2023.
		
		\bibitem{gyorfi2002distribution}
		L{\'a}szl{\'o} Gy{\"o}rfi, Michael Kohler, Adam Krzyzak, Harro Walk, et~al.
		\newblock {\em A distribution-free theory of nonparametric regression},
		volume~1.
		\newblock Springer, 2002.
		
		\bibitem{kohler2021rate}
		Michael Kohler and Sophie Langer.
		\newblock On the rate of convergence of fully connected deep neural network
		regression estimates.
		\newblock {\em The Annals of Statistics}, 49(4):2231--2249, 2021.
	\end{thebibliography}
\end{document}